\documentclass{amsart}
\usepackage{latexsym,amssymb,amsmath,epsfig}
\usepackage{epsfig,graphicx,latexsym}
\usepackage{pstricks,pstricks-add,pst-math,pst-xkey}
\usepackage[latin1]{inputenc} 
\def\proof{{\boldmath $Proof.$}\hskip 0.3truecm}

\newtheorem{lm}{Lemma}[section]
\def\proof{{\it Proof.}\nobreak\\}
\def\qed{\hfill$\Box$ \bigskip}
\newtheorem{tm}{Theorem}[section]

\newtheorem{pro}{Proposition}[section]
\newtheorem{co}{Corollary}[section]

\newtheorem{rem}[tm]{Remark}

\newtheorem{ex}[tm]{Example}

\newcommand{\subgp}[1]{\langle{#1}\rangle}

\begin{document}

\title[Rainbow eulerian multidigraphs and the product of cycles]{Rainbow eulerian multidigraphs and the product of cycles}

\author{S. C. L\'opez}
\address{%
Departament de Matem\`{a}tica Aplicada IV\\
Universitat Polit\`{e}cnica de Catalunya. BarcelonaTech\\
C/Esteve Terrades 5\\
08860 Castelldefels, Spain}
\email{susana@ma4.upc.edu}

\author{F. A. Muntaner-Batle}
\address{Graph Theory and Applications Research Group \\
 School of Electrical Engineering and Computer Science\\
Faculty of Engineering and Built Environment\\
The University of Newcastle\\
NSW 2308
Australia}
\email{famb1es@yahoo.es}
\dedicatory{This paper is dedicated to Miquel Rius-Font for his contribution to graph labelings.}
\date{\today}
\maketitle

\begin{abstract}
An arc colored eulerian multidigraph with $l$ colors is rainbow eulerian if there is an eulerian circuit in which a sequence of $l$ colors repeats. The digraph product that refers the title was introduced by Figueroa-Centeno et al. as follows: let $D$ be a digraph and let $\Gamma$ be a family of digraphs such that $V(F)=V$ for every $F\in \Gamma$. Consider any function $h:E(D)\longrightarrow\Gamma $.
Then the product $D\otimes_{h} \Gamma$ is the digraph with vertex set $V(D)\times V$ and
$((a,x),(b,y))\in E(D\otimes_{h}\Gamma)$ if and only if $ (a,b)\in E(D)$ and $ (x,y)\in E(h (a,b))$.

In this paper we use rainbow eulerian multidigraphs and permutations as a way to characterize the $\otimes_h$-product of oriented cycles. We study the behavior of the $\otimes_h$-product when applied to digraphs with unicyclic components. The results obtained allow us to get edge-magic labelings of graphs formed by the union of unicyclic components and with different magic sums.

\end{abstract}

\section{Introduction}

For the undefined concepts appearing in this paper, we refer the reader to \cite{W}. We denote by $C_n^+$ and by $C_n^-$ the two possible strong orientations of the cycle
$C_n$ and by $\overrightarrow{G}$ any orientation of a graph $G$.
Let $D$ be a digraph, we denote by $D^-$ the reverse of $D$, that is, the digraph obtained from $D$ by reversing all its arcs. According to this notation, it is clear that  $({C_n^+})^-=C_n^-$.

An old result of Good (see for instance, \cite{RobTes05}) states that a weakly connected multidigraph $M$ has an eulerian circuit if and only if, for every vertex, indegree equals outdegree. Let $M$ be an arc labeled eulerian multidigraph with $l$ colors. We say that $M$ is {\it rainbow eulerian} if it has an eulerian circuit in which a sequence of $l$ colors repeats. Similarly, a {\it rainbow circuit} in $M$ is a circuit in $M$ in which a sequence of $l$ colors repeats.

Figueroa-Centeno et al. introduced in \cite{F1} the following generalization of the classical direct product for digraphs. Let $D$ be a digraph and let $\Gamma $ be a
family of digraphs such that $V(F)=V$, for every $F\in\Gamma$. Consider any function $h:E(D)\longrightarrow\Gamma $.
Then the product $D\otimes_{h} \Gamma$ is the digraph with vertex set $V(D)\times V$ and $((a,x),(b,y))\in E(D\otimes_{h
    }\Gamma)$ if and only if $(a,b)\in E(D)$ and $(x,y)\in
    E(h (a,b))$.
The adjacency matrix of $D\otimes_{h} \Gamma$ is obtained by multiplying every $0$ entry of $A(D)$, the adjacency matrix of $D$, by the $|V|\times |V|$ null matrix and every $1$ entry of $A(D)$ by $A(h(a,b))$, where $(a,b)$ is the arc related to the corresponding $1$ entry. Notice that when $h$ is constant, the adjacency matrix of $D\otimes_{h} \Gamma$ is just the classical Kronecker product $A(D)\otimes A(h(a,b))$. When $|\Gamma |=1$, we just write $D\otimes\Gamma$.

A known result in the area (see for instance, \cite{HamImrKlav11}) is that the direct product of two strongly oriented cycles produces copies of a strongly oriented cycle, namely,
\begin{equation}\label{producte_directe_strong_cicles}
    C_m^+ \otimes C_n^+\cong \hbox{gcd}(m,n)C_{\hbox{lcm}(m,n)}^+.
\end{equation}

An extension of (\ref{producte_directe_strong_cicles}) was obtained by Ahmad et al. in \cite{AMR}.

\begin{tm}\cite{AMR}\label{cicleALI}
Let $m, n\in \mathbb{N}$ and consider the product $C_m^+ \otimes_h \{{C_n^+, C_n^- }\}$ where $h:
E(C_m^+)\longrightarrow \{{C_n^+, C_n^- }\}$. Let $g$ be a generator of a cyclic subgroup of
$\mathbb{Z}_n$, namely $\subgp{g}$, such that $|\subgp{g}|=k$. Also let $r<m$ be a positive integer that satisfies the following congruence relation
$$ m-2r\equiv g\ (mod\,\ n).$$

 If the function $h$ assigns   $C_n^-$ to exactly $r$ arcs of $C_m^+$ then the product
$$C_m^+ \otimes_h \{{C_n^+, C_n^- }\}
$$
consists of exactly $n/k$ disjoint copies of a strongly oriented cycle $C_{mk}^+$. In particular if gcd$(g,n)=1$, then
$\subgp{g}=\mathbb{Z}_n$ and if the function $h$ assigns $C_n^-$ to exactly $r$ arcs of $C_m^+$
then $$ C_m^+ \otimes_h \{{C_n^+, C_n^- }\}\cong C_{mn}^+.$$
\end{tm}

In this paper, we study $C_m^+\otimes_h\Gamma$, where $\Gamma$ is a family of $1$-regular digraphs. We characterize this product in terms of rainbow eulerian multidigraphs (Theorem \ref{Ahmad_extension}), which leads us to a further characterization in terms of permutations (Theorem \ref{Ahmad_extension_permutations}). This is the content of Section 2. Section 3 is focused on applications to the product of unicyclic graphs and of graphs whose connected components are unicyclic graphs (Theorems \ref{producte_uni}, \ref{producte_uni_amb_permutacions} and \ref{th_charact_decom_suma_unicyclics}). Finally, in Section 3, we obtain (super) edge-magic labelings of graphs that are union of unicyclic graphs. One of them (Theorem \ref{generalizing_McQ_tm}) generalizes a previous result found in \cite{McQ02}. We also construct families of graphs with an increasing number of possible magic-sums (Theorem \ref{versio_complerta_i_magic_sums}).

\section{Arc colored eulerian multidigraphs obtained from the product}
Let $\Gamma$ be a family of $1$-regular digraphs such that $V(F)=V$, for every $F\in \Gamma$. Consider any function $h:E(C_m^+)\longrightarrow\Gamma $. We denote by $M_h$ the eulerian multidigraph with vertex set $V(M_h)=V$ and arc set $E(M_h)=\cup_{e\in E(C_m^+)}E(h(e))$, including repetitions, that is, if $(x,y)$ is an arc in $r$ different digraphs $h(e)$, $e\in C_m^+$, then $(x,y)$ appears in $M_h$ with multiplicity $r$. An example is shown in Figure 1.

\begin{figure}[ht]
\begin{center}
  \includegraphics[width=150pt]{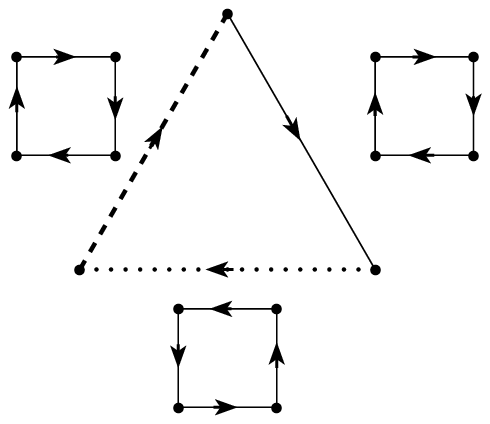}\hspace{1cm}\includegraphics[width=124pt]{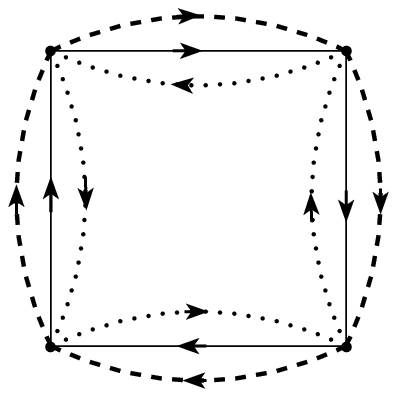}\\
\caption{A function $h:E(C_3^+)\longrightarrow\Gamma $ and the multidigraph $M_h$.}
  \label{Fig1}
  \end{center}
\end{figure}

Assume that we color the arcs of $C_m^+$ using a set of $m$ different colors $\{s_1,s_2,\ldots, s_m\}$, in such a way that, we start by assigning color $s_1$ to some particular arc, and then, we apply the next rule: if the arc $e$ has color $s_i$ assigned then arc $e'$ receives color $s_{i+1}$, where the head of $e$ is the tail of $e'$. We proceed in this way, until all arcs have been colored. We will refer to that coloring as a coloring with color sequence $(s_1,s_2,\ldots, s_m)$. Thus, we obtain an induced arc coloring of the eulerian multidigraph $M_h$, in which the arcs related to $h(e)$ receive the color of $e$, for each $e\in E(C_m^+)$. Figure 1 shows the coloring of the multidigraph $M_h$ induced by a coloring of $C_3^+$ with color sequence (dash, line, dots).

\begin{tm}\label{main}
Let $\Gamma$ be a family of $1$-regular digraphs such that $V(F)=V$ for every $F\in \Gamma$, and $n=|V|$. Consider any function $h:E(C_m^+)\longrightarrow\Gamma $. Then, the digraph $C_m^+\otimes_h\Gamma$ is strongly connected, that is, is a strongly oriented cycle of length $mn$, if and only if, the eulerian multidigraph $M_h$ is rainbow eulerian with color sequence $(s_1,s_2,\ldots, s_m)$, when we consider the arc coloring of $M_h$ induced by a coloring of $C_m^+$ with color sequence $(s_1,s_2,\ldots, s_m)$.
\end{tm}

\proof  By definition of the arc set, the digraph $C_m^+\otimes_h\Gamma$ is $1$-regular. Let $V(C_m^+)=\{a_1,a_2,\ldots, a_m\}$ with $E(C_m^+)=\{(a_i,a_{i+1})\}_{i=1}^{m-1}\cup \{(a_m,a_1)\}$, and with the arc $(a_i,a_{i+1})$ colored  $s_i$, for $i=1,2,\ldots, m-1$, and with the arc $(a_m,a_{1})$ colored  $s_m$.
Assume first that $C_m^+\otimes_h\Gamma$ is a strongly oriented cycle of length $mn$. Let $v_1v_2\ldots v_{mn}$ be a hamiltonian path in $C_m^+\otimes_h\Gamma$. Without loss of generality assume that $\pi_1(v_1)=a_1$, where $\pi_1:V(C_m^+\otimes_h\Gamma)\rightarrow V(C_m^+)$ is the natural projection $\pi_1(a,x)=a$. By definition of the $\otimes_h$-product, we have that $\pi_1(v_r)=a_i$ if and only if $r\equiv i$ (mod $m$). Thus, for every $r$ and every $i$ with, $1\le r\le mn-1$, $1\le i\le m$ and $r\equiv i$ (mod $m$), the arc $(v_r,v_{r+1})$ is related to the arc $(a_i,a_{i+1})\in E(C_m^+)$ and an arc $(x,y)\in h(a_i,a_{i+1})$, where if $i=m$, the expression  $(a_i,a_{i+1})$ should be replaced by $(a_m,a_1)$. Since every arc $h(a_i,a_{i+1})$ receives color $s_i$, we have that the arc $(v_r,v_{r+1})$ is related to the arc $(x,y)$ of $M_h$ colored with $s_i$. Therefore, the multidigraph $M_h$ contains a rainbow eulerian circuit, namely $\pi_2(v_1),\pi_2(v_2),\ldots, \pi_2(v_{mn})$ where $\pi_2(a,x)=x$, with color sequence $(s_1,s_2,\ldots, s_m)$. The converse is similar. Assume that $x_1x_2\ldots x_{mn}x_1$ is a rainbow eulerian circuit with color sequence $(s_1,s_2,\ldots,s_m)$. Without loss of generality, assume that $x_1x_2$ is colored using color $s_1$. Then, $x_rx_{r+1}\in E(h(a_ia_{i+1}))$, for every $r\equiv i$ (mod $m$). Hence, if we let $v_r=(a_i,x_r)$, where $ i\equiv r$ (mod $m$), we get a hamiltonian cycle $v_1v_2\ldots v_{mn}v_1$ in $C_m^+\otimes_h\Gamma$.
\qed

For instance, the digraph $C_3^+\otimes_h\Gamma$, where $h$ is the function defined in Figure 1, is a strongly oriented cycle of length $12$, since, as Figure 2 shows, we can find a rainbow eulerian circuit with color sequence (dash, line, dots). The arc labels denote the place that each arc occupies in the rainbow eulerian circuit when we start to follow the circuit with the arc labeled $1$.

\begin{figure}[ht]
\begin{center}
  \includegraphics[width=124pt]{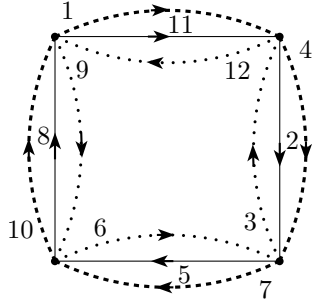}\\
\caption{A rainbow eulerian circuit of $M_h$ with color sequence (dash, line, dots).}
  \label{Fig2}
  \end{center}
\end{figure}

From the proof of Theorem \ref{main} we also obtain the next result, which is a generalization of Theorem \ref{cicleALI} in terms of the existence of rainbow circuits.

\begin{tm}\label{Ahmad_extension}
Let $\Gamma$ be a family of $1$-regular digraphs such that $V(F)=V$ for every $F\in \Gamma$. Consider any function $h:E(C_m^+)\longrightarrow\Gamma $. Then, every rainbow circuit in $M_h$ with color sequence $(s_1,s_2,\ldots, s_m)$ corresponds to a strongly connected component of $C_m^+\otimes_h\Gamma$, when we consider the coloring of $M_h$ induced by a coloring of $C_m^+$ with color sequence $(s_1,s_2,\ldots, s_m)$.
\end{tm}

\begin{rem}
Let $h,h':E(C_m^+)\rightarrow \Gamma$ be two functions with Im $h=$ Im $h'$, when  Im $h$ and Im $h'$ are considered to be multisets. The next example shows that the relation $C_m^+\otimes_h\Gamma\cong C_m^+\otimes_{h'}\Gamma$ does not hold, in general.
\end{rem}

\begin{ex}\label{Exemple_h_h'}
Let $\Gamma=\{F_i\}_{i=1}^3\cup\{F^-_1\}$, where $F_1,F_2$ and $F_3$ are the digraphs that appear in Figure \ref{Fig3}. Assume that $E(C_4^+)=\{e_i\}_{i=1}^4$ and that the head of $e_i$ is the tail of $e_{i+1}$, for $i=1,2,3$. Consider the functions $h,h':E(C_4^+)\rightarrow \Gamma$ defined by:
$h(e_1)=h'(e_1)=F_1; h(e_2)=h'(e_4)=F_2; h(e_3)=h'(e_3)=F_1^-; h(e_4)=h'(e_2)=F_3.$
Then,
$$C_4^+\otimes_h\Gamma\cong C_{16}^++ C_8^+\hspace{0.5cm} \hbox{and}\hspace{0.5cm} C_4^+\otimes_{h'}\Gamma\cong C_4^++ C_{20}^+.$$

\begin{figure}[ht]
\begin{center}
  \includegraphics[width=316pt]{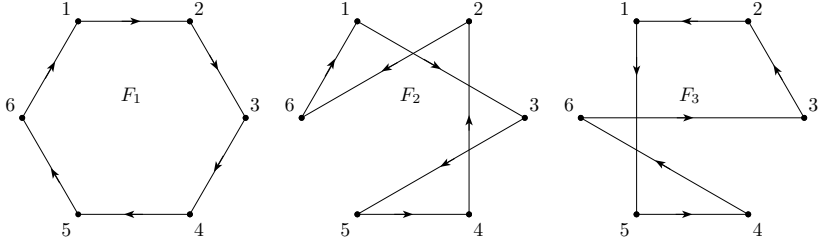}\\
\caption{The digraphs $F_1$, $F_2$ and $F_3$.}
  \label{Fig3}
  \end{center}
\end{figure}
\end{ex}





The next lemma, although it is fairly simple, will prove to be useful in order to obtain Corollary \ref{co_weakly_coonected_components}.

\begin{lm}\label{reverse}
Let $K_2^+$ and $K_2^-$ be the two possible orientations of $K_2$. Then, the reverse of $K_2^+\otimes C_m^+$ is the digraph $K_2^-\otimes C_m^-$. In particular,
$$\hbox{und}(K_2^+\otimes C_m^+)\cong\hbox{und}(K_2^-\otimes C_m^-).$$
\end{lm}

\proof
Let $D=(K_2^+\otimes C_m^+)^-$. By definition $V(D)=V(K_2^+)\times V(C_m^+)$, that is, $V(D)=V(K_2^-)\times V(C_m^-)$, and $((a,x),(b,y))\in E(D)$ if and only if $((b,y),(a,x))\in E(K_2^+\otimes C_m^+)$. Thus, $((a,x),(b,y))\in E(D)$ if and only if, $(b,a)\in E(K_2^+)$ and $(y,x)\in V(C_m^+)$, that is, if and only if,
 $(a,b)\in E(K_2^-)$ and $(x,y)\in V(C_m^-)$. Hence, $((a,x),(b,y))\in E(D)$ if and only if, $((a,x),(b,y))\in E(K_2^-\otimes C_m^-)$. This proves the main statement. Therefore, the final statement trivially holds.
\qed

\begin{co}\label{co_weakly_coonected_components}
Let $\overrightarrow{C}_m$ be any orientation of a cycle $C_m$, and let $\Gamma$ be a family of $1$-regular digraphs such that $V(F)=V$ for every $F\in \Gamma$. Assume that $h:E(\overrightarrow{C}_m)\rightarrow \Gamma$ is any function. We consider the function $h^*:E(C_m^+)\rightarrow \Gamma\cup \Gamma^-$, where $\Gamma^-=\{F:\ F^-\in \Gamma\}$, defined by:
$h^*(x,y)=h(x,y)$ if $(x,y)\in E(\overrightarrow{C}_m)$, or,  $h^*(x,y)=h(x,y)^-$ if $(y,x)\in E(\overrightarrow{C}_m)$.

Then, every rainbow  circuit in $M_{h^*}$ with color sequence $(s_1,s_2,\ldots, s_m)$ corresponds to a weakly connected component of $\overrightarrow{C}_m\otimes_h\Gamma$, when we consider the coloring of $M_{h^*}$ induced by a coloring of $C_m^+$ with color sequence $(s_1,s_2,\ldots, s_m)$ .
\end{co}

\proof The key point is that und$(\overrightarrow{C}_m\otimes_h\Gamma)$=und$(C_m^+\otimes_{h^*}(\Gamma\cup \Gamma^-))$, since by Lemma \ref{reverse}, the digraph $\overrightarrow{C}_m\otimes_h\Gamma$ can be obtained from $C_m^+\otimes_{h^*}(\Gamma\cup \Gamma^-)$ by reversing the arcs related to the arcs in $E(C_m^+)\setminus E(\overrightarrow{C}_m)$. Then, the result holds by Theorem \ref{Ahmad_extension}.\qed

\vspace{1cm}


\subsection{The induced product of permutations}
A permutation $\pi$ is a bijective mapping $\pi: \{i\}_{i=1}^n\rightarrow \{i\}_{i=1}^n$. It is well known that a common way to describe permutations is by means of the product of mutually disjoint strongly oriented cycles, in which $(i,j)$ is an arc of an oriented cycle if and only if $\pi (i)=j$. Thus, every $1$-regular digraph on $V$ is identified with a permutation on the elements of $V$. This idea allows us to express Theorem \ref{Ahmad_extension} in terms of permutations.

Let $V(C_m^+)=\{a_1,a_2,\ldots, a_m\}$ with $E(C_m^+)=\{(a_i,a_{i+1})\}_{i=1}^{m-1}\cup \{(a_m,a_1)\}$ and let $\Gamma$ be a family of $1$-regular digraphs such that $V(F)=V$, for every $F\in \Gamma$, and $|V|=n$. Consider any function $h:E(C_m^+)\longrightarrow\Gamma $. Then, if we identify each element on $\Gamma$ with a permutation on $V$, we can consider the product of permutations $P_h=h(a_ma_1)\cdots h(a_2a_3)\cdot h(a_1a_2 )\in \mathfrak{S}_n,$ where $ \mathfrak{S}_n$ is the set of all permutations on the set $V$.


\begin{ex}
Let $h,h':E(C_4^+)\rightarrow \Gamma$ be the functions introduced in Example \ref{Exemple_h_h'}. Then
\begin{eqnarray*}
P_h&=& h(e_4)h(e_3)h(e_2)h(e_1)=(1\ 5\ 4\ 6\ 3\ 2)(1\ 5\ 4\ 6\ 3\ 2)(1\ 6\ 5\ 4\ 3\ 2)(1\ 3\ 5\ 4\ 2\ 6)(1\ 2\ 3\ 4\ 5\ 6) \\
  &=& (1\ 4\ 2\ 6)(3\ 5).\\
P_{h'}&=&h'(e_4)h'(e_3)h'(e_2)h'(e_1) = (1\ 3\ 5\ 4\ 2\ 6)(1\ 6\ 5\ 4\ 3\ 2)(1\ 5\ 4\ 6\ 3\ 2)(1\ 2\ 3\ 4\ 5\ 6)\\
  &=& (1)(2\ 3\ 4\ 5\ 6).\\
\end{eqnarray*}
Thus,
the products $(1\ 4\ 2\ 6)(3\ 5)$ and $(1)(2\ 3\ 4\ 5\ 6)$ are the disjoint cyclic decompositions (which are uniquely defined up to ordering) of $P_h$ and $P_{h'}$, respectively.
\end{ex}

The following result is a generalization of Theorem \ref{cicleALI} in terms of permutations.
\begin{tm}\label{Ahmad_extension_permutations}
Let $\Gamma$ be a family of $1$-regular digraphs such that $V(F)=V$ for every $F\in \Gamma$. Consider any function $h:E(C_m^+)\longrightarrow\Gamma $. Then, every cycle in the disjoint cyclic decomposition of $P_h$ corresponds to a strongly connected component of $C_m^+\otimes_h\Gamma$. Moreover if $\sigma$ is a cycle in the disjoint cyclic decomposition of $P_h$ then the corresponding strongly component of $C_m^+\otimes_h\Gamma$ has length $m$ times the length of $\sigma$. That is,
$$C_m^+\otimes_h\Gamma\cong C^+_{m|\sigma_1|}+C^+_{m|\sigma_2|}+\ldots +C^+_{m|\sigma_s|},$$
where, the product $\sigma_1\cdots \sigma_{s}$ is the disjoint cyclic decomposition of $P_{h}$ and $|\sigma_j|$ denotes the length of $\sigma_j$, for $j=1,2,\ldots, s$.
\end{tm}

\proof Consider the coloring of $C_m^+$ with color sequence $(s_1,s_2,\ldots, s_m)$ that assigns color $s_1$ to $(a_1,a_{2})$. By Theorem \ref{Ahmad_extension}, every rainbow  circuit in $M_h$ with color sequence $(s_1,s_2,\ldots, s_m)$ corresponds to a strongly connected component of $C_m^+\otimes_h\Gamma$, when we consider the coloring of $M_h$ induced by the coloring of $C_m^+$. Since every rainbow circuit in $M_h$ can be obtained following a cycle in the product of permutations $h(a_ma_1)\cdots h(a_2a_3)\cdot h(a_1a_2 )$, we get that every
cycle in the disjoint cyclic decomposition of $P_h$ corresponds to a strongly component of $C_m^+\otimes_h\Gamma$.
Moreover, every element in this cycle represents the tail of the arc colored with $s_1$. Hence, the length of the strongly connected component is $m$ times the length of the cycle.
\qed

\begin{rem}
Note that, as a Corollary of the previous theorem we get Theorem \ref{cicleALI}. The reason is that, using the notation introduced in Theorem \ref{cicleALI}, together with the idea of permutations, the product
$P_h$ is a sequence of $m$ factors of the form $\{\sigma, \sigma^-\}$, where $r$ of them are $\sigma^-$. Thus, assuming that $m\ge 2r$ (the other case is similar), we get
$P_h=\sigma^{m-2r}$. Hence, if $|\subgp{m-2r}|=k$ in $Z_n$ then $P_h$ decomposes into $n/k$ disjoint cycles of length $k$. Therefore, by Theorem \ref{Ahmad_extension_permutations}, we obtain that $$C_m^+ \otimes_h \{{C_n^+, C_n^- }\}\cong \frac{n}{k}C^+_{mk}.
$$
\end{rem}

In order to conclude this section, we want to observe that so far, we have been using the theory of
permutations in order to describe the behavior of the $\otimes_h$-product for $1$-regular digraphs.
However, and although it is not the goal of this paper, it is worth to notice that this product can also
be used in order to understand properties of the permutations.

\section{(Di)graphs with unicyclic connected components}

Let $G=G_1+G_2+\ldots +G_l$ be a simple graph in which every $G_i$ is a connected component that contains exactly one cycle. Let $m_i$ be the length of the cycle in $G_i$ and let $\{a_j^i\}_{j=1}^{m_i}$ and $\{a_j^ia_{j+1}\}_{j=1}^{m_i-1}\cup \{a_m^ia_1^i\}$ be the vertex and edge sets of the cycle in $G_i$, respectively. Let $T_j^i$ be the tree attached at vertex $a_j^i$, where $|V(T_j^i)|\ge 1$. If $w_j^i\in V(T_j^i)$ is the vertex identified with $a_j^i$ then, we denote $G_i$ as follows:
$G_i=(T_1^i (w_1^i), T_2^i (w_2^i),\ldots, T_{m_i}^i (w_{m_i}^i)).$

Using this notation, the graph $G$ can be described as a succession of tuples.
Suppose that a component $G_i$ is of the form:
$$G_i= (T_1^i (w_1^i), T_2^i (w_2^i),\ldots, T_{m_i}^i (w_{m_i}^i), T_1^i (w_1^i), \ldots, T_{m_i}^i (w_{m_i}^i),\ldots, T_1^i (w_1^i), \ldots, T_{m_i}^i (w_{m_i}^i))
$$
where the sequence $(T_1^i (w_1^i), T_2^i (w_2^i),\ldots, T_{m_i}^i (w_{m_i}^i))$ repeats $k$ times. Then, we say that $G_i$ is a periodic component of multiplicity $k$. We denote such a component
$$G_i=(T_1^i (w_1^i), T_2^i (w_2^i),\ldots, T_{m_i}^i (w_{m_i}^i)) ^k.$$

An example of this notation appears in Figure \ref{Fig4}.

\begin{figure}[h]
\begin{center}
  \includegraphics[width=153pt]{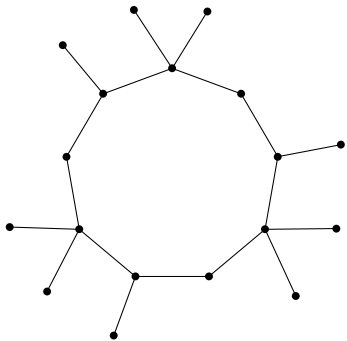}\\
  \caption{The graph $G=(P_1(w_1),P_2(w_2),P_3(w_3))^3$, where $P_s$ is the path of order $s$ and $w_3$ is the central vertex of $P_3$.}\label{Fig4}
  \end{center}
\end{figure}

A useful result in order to study the union of acyclic graphs is due to Figueroa et al. \cite{F1}.

\begin{tm}\cite{F1}\label{product_of_tree}
Let $F$ be an acyclic graph and let $\Sigma_n$ be the set of $1$-regular digraphs of order $n$. Consider any function $h:E(\overrightarrow{F})\rightarrow \Sigma_n$. Then, $\overrightarrow{F}\otimes_h\Sigma_n=n\overrightarrow{F}.$
\end{tm}

The following lemma is an easy observation.

\begin{lm}\label{components}
Let $D=D_1+D_2+\ldots+D_l$ be a digraph where $D_1$, $D_2$,\ldots, $D_l$ are the weakly connected components of $D$. Consider any function $h:E(D)\rightarrow \Gamma$, where $\Gamma$ is a family of digraphs such that $V(F)=V$ for every $F\in \Gamma$. Let $h_i=h_{|E(D_i)}:E(D_i)\rightarrow \Gamma$ be the restriction of $h$ over $E(D_i)$, for each $i=1,2,\ldots, l$. Then, $$D\otimes_h\Gamma\cong (D_1\otimes_{h_1}\Gamma)+\ldots +(D_l\otimes_{h_l}\Gamma).$$
\end{lm}

As a corollary of Theorem \ref{cicleALI} and Theorem \ref{product_of_tree}, together with Lemma \ref{components} we get the next result.

\begin{tm}\label{producte_uni}
Let $G=G_1+G_2+\ldots +G_l$ be a simple graph where each $G_i$ is of the form $G_i=(T_1^i (w_1^i), T_2^i (w_2^i),\ldots, T_{m_i}^i (w_{m_i}^i))$ and each $T_j^i$ is a tree. Consider $D=D_1+D_2+\ldots+D_l$ an oriented graph obtained from $G$ by considering strong orientations of all its cycles, where und$(D_i)=G_i$. Let $h:E(D)\rightarrow \{C_n^+, C_n^- \}$ be any function such that the restriction of $h$ over $E(D_i)$ assigns $C_n^-$ to exactly $r_i$ arcs of its cicle, for each $i=1,2,\ldots, l$. Then,
$$und(D\otimes_h\Gamma)\cong \frac{n}{k_1}(T_1^1 (w_1^1), T_2^1 (w_2^1),\ldots, T_{m_1}^1 (w_{m_1}^1))^{k_1}+\ldots +\frac{n}{k_l}(T_1^l(w_1^l), T_2^l (w_2^l),\ldots, T_{m_l}^l (w_{m_l}^l))^{k_l},$$
where, for each $i=1,2,\ldots, l$, $k_i$ is the order of the subgroup of $\mathbb{Z}_n$ generated by $m_i-2r_i$.
\end{tm}

\proof By Lemma \ref{components}, it is enough to prove that if $G$ is of the form $(T_1 (w_1), T_2 (w_2),\ldots, T_{m} (w_{m}))$ then und$(D\otimes_h\Gamma)\cong n/k(T_1 (w_1), T_2 (w_2),\ldots, T_{m} (w_{m}))^{k}$, where $D$ is an oriented graph obtained from $G$ by considering a strong orientation of its cycle, namely $C_m^+$ and $h:E(D)\rightarrow \{C_n^+, C_n^- \}$ is any function. By Theorem \ref{cicleALI}, if the function $h_{|E(C_m^+)}:E(C_m^+))\rightarrow \{C_n^+, C_n^- \}$ assigns $C_n^-$ to exactly $r$ arcs of  $C^+_m$ then $(C_m^+\otimes_{h_{|E(C_m^+)}}\{C_n^+, C_n^- \})\cong n/k C^+_{mk},$ where $k$ is the order of the subgroup of $Z_n$ generated by $m-2r$. Let $\overrightarrow{F}$ be the oriented digraph obtained from $T_1+ T_2 +\ldots+ T_{m}$ by considering the orientation induced by $D$. Assume that $V(C_n^+)=V(C_n^-)=\{i\}_{i=1}^n$. By Theorem \ref{product_of_tree}, we obtain that $\overrightarrow{F}\otimes_{h_{|E(\overrightarrow{F})}}\{C_n^+, C_n^- \}=n\overrightarrow{F}.$ Moreover, since each vertex $a_j\in V(C_m^+)$ is identified with a vertex $w_j\in V(T_j)$, we obtain that $(a_j,i)$ is identified with $(w_j,i)$, for each $i=1,2,\ldots, n$ and for each $j=1,2,\ldots,m$. Therefore, the desired result has been reached.\qed

An extension of the previous result can be obtained by considering Theorem \ref{Ahmad_extension_permutations}, Theorem \ref{product_of_tree} and Lemma \ref{components}.
\begin{tm}\label{producte_uni_amb_permutacions}
Let $G=G_1+G_2+\ldots +G_l$ be a simple graph where each $G_i$ is of the form $G_i=(T_1^i (w_1^i), T_2^i (w_2^i),\ldots, T_{m_i}^i (w_{m_i}^i)).$ Consider an oriented graph $D=D_1+D_2+\ldots+D_l$ obtained from $G$ by considering strong orientations of all its cycles, where und$(D_i)=G_i$. Let $\Gamma$ be a family of $1$-regular digraphs such that $V(F)=V$ for every $F\in \Gamma$. Consider any function $h:E(D)\longrightarrow\Gamma $. Denote by $C^+_{m_i}$ the cycle of $D_i$ and $h^c_i=h_{|E(C^+_{m_i})}:E(C^+_{m_i})\rightarrow \Gamma$ the restriction of $h$ over $E(C^+_{m_i})$.
Then,
$$und(D\otimes_h\Gamma)\cong \sum_{i=1}^l\sum_{j=1}^{s_i}(T_1^i (w_1^i), T_2^i (w_2^i),\ldots, T_{m_i}^i (w_{m_i}^i))^{|\sigma_j^i|},$$
where, for each $i=1,2,\ldots, l$, the product $\sigma_1^i\cdots \sigma_{s_i}^i$ is the disjoint cyclic decomposition of $P_{h^c_i}$ and $|\sigma_i^j|$ denotes the length of $\sigma_j^i$.
\end{tm}

\proof The proof is similar to the proof of Theorem \ref{producte_uni}. The only difference is that we use Theorem \ref{Ahmad_extension_permutations} instead of Theorem \ref{cicleALI}.\qed

One of the goals in \cite{LopMun13a} was the characterization of the existence of a nontrivial decomposition of a graph in terms of the $\otimes_h$-product. The next result is a contribution in this direction when we apply the product over and over again to digraphs. Moreover, this result can be though as the converse of Theorem \ref{producte_uni}.

\begin{tm}\label{th_charact_decom_suma_unicyclics}
Let $l,m,n$ and $s$ be nonnegative integers, such that, $m\ge 1$, $n\ge 3$ is odd and if $m$ is odd then $m\ge n$. Let $G=(T_1(w_1),T_2(w_2),\ldots, T_m(w_m))$ be a unicyclic graph with a cycle of length $m$ and $D_1$ be an oriented graph obtained from $G$ by considering a strong orientation of its cycle. Assume that $(a_0^l,a_1^l,\ldots, a_l^l)$ is any sequence of nonnegative integers such that, $a_0^0=n$ when $l=0$,
$$\sum_{i=0}^{l}\frac{a_{i}^l}{n^{l-i}}=n \ \hbox{and}\ \sum_{i=0}^k\frac{a_{i}^l}{n^{k-i+1}}\in \mathbb{Z}^+,\ \hbox{for each}\ k=0,1,l-1,\ \hbox{when}\ l\ge 1.$$
Then, there exists a sequence of digraphs $D_i$ and a sequence of functions $h_i: E(D_i)\rightarrow \{C_n^+,C_n^-\}$,
 $i=1,\ldots, l+s+1$, such that $D_{i+1}=D_i\otimes_{h_i}\{C_n^+,C_n^-\}$,
$$G^{n^{s}}\cong \hbox{und}(D_{s+1})\quad \hbox{and}\quad \sum_{i=0}^la_i^l G^{n^{s+i}}\cong \hbox{und}(D_{l+s+2}).$$
\end{tm}

\proof If $m$ is odd we define $r_i=(mn^{i-1}-1)/2$. Otherwise, we define $r_i=(mn^{i-1}-2)/2$. In both cases, the subgroup of $\mathbb{Z}_n$ generated by $mn^{i-1}-2r_i$ is $\mathbb{Z}_n$. Similarly, if $m$ is even we define $r'_i=mn^{i-1}/2$. Otherwise, we define $r'_i=(mn^{i-1}+n)/2$. Thus, the
subgroup of $\mathbb{Z}_n$ generated by $mn^{i-1}-2r'_i$ is the trivial subgroup.

We construct $h_i$ and $D_{i+1}$ recursively as follows, for $i\ge 1$. Let $h_i: E(D_i)\rightarrow \{C_n^+,C_n^-\}$ be any function such that ${h_i}_{|E(C_{mn^{i-1}}^+)}:E(C_{mn^{i-1}}^+)\rightarrow \{C_n^+,C_n^-\}$ assigns $C_n^-$ to exactly $r_i$ arcs of $C_{mn^{i-1}}^+$. Define $ D_{i+1}=D_i\otimes_{h_i}\{C_n^+,C_n^-\}$. By Theorem \ref{producte_uni}, we obtain that und$(D_{s+1})\cong G^{n^s}$. Thus, the first part of the statement holds.

Let $h_{s+1}: E(D_{s+1})\rightarrow \{C_n^+,C_n^-\}$ be any function such that ${h_{s+1}}_{|E(C_{mn^s}^+)}:E(C_{mn^s}^+)\rightarrow \{C_n^+,C_n^-\}$ assigns $C_n^-$ to exactly $r'_{s+1}$ arcs of $C_{mn^s}^+$.  By Theorem \ref{producte_uni}, we obtain that und$(D_{s+2})\cong nG^{n^s}$. Hence, the result holds for $l=0$. From now on, assume that $l\ge 1$. We will prove by induction on $l$ that there exists a sequence of digraphs $D_{i+1}$ and a sequence of functions $h_i: E(D_i)\rightarrow \{C_n^+,C_n^-\}$,
 $i=s+2,\ldots, l+s+1$, such that $D_{i+1}=D_i\otimes_{h_i}\{C_n^+,C_n^-\}$ and
\begin{equation}\label{formula _D}
\hbox{und}(D_{l+s+2})\cong \sum_{k=1}^{l} \left(n\sum_{t=1}^kj_t^{l-k+t}-\sum_{t=1}^{k-1}j_{t}^{l-k+t+1}\right)
G^{n^{k+s-1}}+\left(n-\sum_{k=1}^lj_k^k\right)G^{n^{l+s}},
\end{equation}
where $\{j_1^1, j_1^2,j_2^2,\ldots,
j_1^l,j_2^l,\ldots, j_l^l\}$ is a set of nonnegative integers, and we let $\sum_{t=1}^{0}\alpha (t)=0$, for any expression $\alpha (t)$ depending on $t$.

Let $h_{s+2}:E(D_{s+2})\rightarrow \{C_n^+,C_n^-\}$ be any function that assigns $C_n^-$ to exactly $r'_{s+1}$ arcs of $C_{mn^{s}}^+$ to $j_1^{1}$ weakly connected components with a cycle of length $mn^{s}$ and that assigns $C_n^-$ to exactly $r_{s+1}$ arcs of $C_{mn^s}^+$ of the remaining weakly connected components with a cycle of length $mn^{s}$. Then, by Theorem \ref{producte_uni}, we have that:
$
  \hbox{und}(D_{s+3}) \cong  nj_1^{1}G^{n^{s}}+\left(n-j_{1}^{1}\right)G^{n^{s+1}}
$. Thus, formula (\ref{formula _D}) holds for $l=1$.
Suppose now that these sequences of digraphs and functions exist for $l$. We have to prove that also exist for $l+1$.
Let $h_{l+s+2}:E(D_{l+s+2})\rightarrow \{C_n^+,C_n^-\}$ be any function that assigns (i) $C_n^-$ to exactly $r'_{s+k}$ arcs of $C_{mn^{s+k-1}}^+$ to $j_k^{l+1}$ weakly connected components with a cycle of length $mn^{s+k-1}$  and (ii) $C_n^-$ to exactly $r_{s+1}$ arcs of $C_{mn^s}^+$ of the remaining weakly connected components with a cycle of length $mn^{s+k-1}$, for $k=1,2,\ldots, l$. Then, by Theorem \ref{producte_uni} and the induction hypothesis, we have that:
\begin{eqnarray*}
  \hbox{und}(D_{l+s+3}) &\cong & \sum_{k=1}^{l}nj_k^{l+1}G^{n^{k+s-1}}+ \sum_{k=1}^{l}\left(n\sum_{t=1}^kj_t^{l-k+t}-\sum_{t=1}^{k-1}j_{t}^{l-k+t+1}-j_k^{l+1}\right)G^{n^{k+s}}\\
   &+& nj_{l+1}^{l+1} G^{n^{l+s}}+\left(n-\sum_{k=1}^lj_k^k-j_{l+1}^{l+1}\right)G^{n^{l+s+1}}\\
   &=& nj_{1}^{l+1} G^{n^{s}}+ \sum_{k=2}^{l+1}nj_k^{l+1}G^{n^{k+s-1}}+ \sum_{k=2}^{l+1}\left(n\sum_{t=1}^{k-1}j_t^{l-(k-1)+t}-\sum_{t=1}^{k-1}j_{t}^{l-(k-1)+t+1}\right)G^{n^{k+s-1}}\\
   &=& \sum_{k=1}^{l+1} \left(n\sum_{t=1}^kj_t^{l+1-k+t}-\sum_{t=1}^{k-1}j_{t}^{l+1-k+t+1}\right)
G^{n^{k+s-1}}+\left(n-\sum_{k=1}^{l+1}j_k^k\right)G^{n^{l+s+1}},
\end{eqnarray*}
which proves that formula (\ref{formula _D}) also holds for $l+1$.
Now, we are ready to finish the proof. Let $(a_0^l,a_1^l,\ldots, a_l^l)$ be any sequence of nonnegative integers such that,
\begin{equation}\label{equation_a_l^l}
   a_l^l=n-\sum_{k=1}^lj_k^k\quad \hbox{and}\quad a_{k-1}^{l}=n\sum_{t=1}^kj_t^{l-k+t}-\sum_{t=1}^{k-1}j_{t}^{l-k+t+1}, \ k=1,2,\ldots, l.
\end{equation}
Then,
$$ \sum_{k=0}^la_k^l G^{n^{k+s}}\cong \hbox{und}(D_{l+s+2}).$$

An easy check shows that we can find a set of nonnegative integers $\{j_1^1, j_1^2,j_2^2,\ldots,
j_1^l,j_2^l,\ldots, j_l^l\}$ that is a solution of (\ref{equation_a_l^l}) if $$\sum_{i=0}^{l}\frac{a_{i}^l}{n^{l-i}}=n \ \hbox{and}\ \sum_{i=0}^k\frac{a_{i}^l}{n^{k-i+1}}\in \mathbb{Z}^+,\ \hbox{for each}\ k=0,1,l-1,\ \hbox{when}\ l\ge 1.$$
Therefore, the result follows.\qed
\section{Super edge-magic labelings of (di)graphs with unicyclic connected components}
Through this section, we use the term graph to mean simple graph. That is to say, the graph considered do not contain loops nor multiple edges. Let $G=(V,E)$ be a $(p,q)$-graph, that is a graph with $|V|=p$ and $|E|=q.$ Kotzig and Rosa introduced in \cite{K1} the concept of edge-magic labeling. A bijective function $f:V\cup E\longrightarrow \{i\}_{i=1}^{p+q}$ is an {\it edge-magic labeling}
of $G$ if there exists an integer $k$ such that the sum $f(x)+f(xy)+f(y)=k$ for all $ xy\in E$. A graph that admits an edge-magic labeling is called an {\it edge-magic graph}, and $k$ is called the {\it valence}, the {\it magic sum} \cite{Wa} or the {\it magic weight} \cite{BaMi} of the labeling. In 1998, Enomoto
el al. \cite{E} defined the concepts of super edge-magic graphs
and super edge-magic labelings. A {\it super edge-magic labeling} is an edge-magic labeling that satisfies the extra condition $f(V)=\{i\}_{i=1} ^{p}$. It is worthwhile mentioning that an equivalent  labeling had already appeared in the literature in
1991 under the name of {\it strongly indexable labeling} \cite{AH}.
 A graph that admits a (super) edge-magic labeling is called a {\it (super) edge-magic graph}.
Super edge-magic graphs and labelings are of great interest in the world of graph labeling since they constitute a powerful link among different types of labelings. See \cite{F2} and more recently \cite{LopMunRiu8}.

From now on, let $\mathcal{S}_n^k$ denote the set of all super edge-magic labeled digraphs with order and size equal to $n$ and magic sum $k$, where
each vertex takes the name of the label that has been assigned to it. Then, the following result can be found en \cite{LopMunRiu8}.

\begin{tm}\cite{LopMunRiu8}\label{super}
Let $D$ be a (super) edge-magic digraph and let $h:E(D)\longrightarrow \mathcal{S}_n^k$ be any function. Then $und(D\otimes _h \mathcal{S}_n^k)$ is (super) edge-magic.
\end{tm}

The key point in the proof, see also \cite{LopMunRiu8}, is to rename the vertices of $D$ and each element of $\mathcal{S}_n^k$ after the labels of their corresponding (super) edge-magic labeling $f$ and their super edge-magic labelings respectively. Then define the labels of the product as follows:
(i) the vertex $(i,j)\in V(D\otimes_h \mathcal{S}_n^k)$ receives the label: $n(i-1)+j$ and
(ii) the arc $((i,j),(i',j'))\in E (D\otimes_h \mathcal{S}_n^k)$ receives the label: $n(e-1)+k+n-(j+j')$, where  $e$ is the label of  $(i,i')$ in $D$.
Thus, for each arc $((i,j),(i',j'))\in E (D\otimes_h \mathcal{S}_n^k)$, coming from an arc $e=(i,i')\in E(D)$ and an arc $(j,j')\in E(h(i,i'))$, the sum of labels is constant and equals to:
$
    n(i+i'+e-3)+k+n.
$
 That is, $n(\sigma_f-3)+k+n,$ where $\sigma_f$ denotes the magic sum of the labeling $f$ of $D$.
Therefore, we obtain the following proposition.

\begin{pro}
\label{Pro_induced_labeling}
Let $\check{f}$ be the edge-magic labeling of the graph und$(D\otimes_h \mathcal{S}_n^k)$ obtained in Theorem \ref{super} from a labeling $f$ of $D$. Then the magic sum of $\check{f}$, $\sigma_{\check{f}}$, is given by the formula

\begin{equation}\label{induced_labeling}
\sigma_{\check{f}}=n(\sigma_f-3)+n+k,
\end{equation}
where $\sigma_f$ is the magic sum of $f$.
\end{pro}

\begin{co}\label{valences_differ}
Let $D$ be an edge-magic digraph and assume that there exist two edge-magic labelings of $D$, $f$ and $g$, such that $\sigma_f\neq\sigma_g$. If we denote by $\check{f}$ and $\check{g}$ the edge-magic labelings of the graph und$(D\otimes_h \mathcal{S}_n^k)$ when using the edge-magic labelings $f$ and $g$ of $D$ respectively, then we get
$$|\sigma_{\check{f}}-\sigma_{\check{g}}|\ge 3.$$
\end{co}

\begin{proof} Since $\sigma_f\neq\sigma_g$, we get the inequality $|\sigma_f-\sigma_g|\ge 1$.
Thus, using (\ref{induced_labeling}), we obtain that
$|\sigma_{\check{f}}-\sigma_{\check{g}}|=|n(\sigma_f- \sigma_g)|\ge 3.$
\end{proof}

Using a technique introduced in \cite{McQ01}, McQuillan proved the following. Let $n$ be a positive integer and let $G=C_{m_1}+\ldots +C_{m_l}$ be a disjoint union of cycles. Let $I=\{1,2,\ldots,l\}$ and $J$ be any subset of $I$. Finally, denote by $G_J=(\sum_{j\in J}nC_{m_j})+(\sum_{i\in I\setminus J}C_{nm_i})$.

\begin{tm}\cite{McQ02}\label{MacQ_tm}
Let $n$ be an odd positive integer and let $G=C_{m_1}+\ldots +C_{m_l}$ be a disjoint union of cycles. Consider $I=\{1,2,\ldots,l\}$ and $J$ any subset of $I$. If $G$ has an edge-magic labeling with magic constant $h$ then $G_J$ has edge-magic labelings with magic constants $3(n-1)(m_1+m_2+\ldots +m_l)+h$ and $nh-3(n-1)/2$.
\end{tm}

Using Theorems \ref{producte_uni} and \ref{super}, we generalize, in some sense, Theorem \ref{MacQ_tm}.

\begin{tm}\label{generalizing_McQ_tm}
Let $n$ be an odd positive integer and let $G=G_1+G_2+\ldots +G_l,$ where each $G_i$ is a unicyclic graph of the form $G_i=(T_1^i (w_1^i), T_2^i (w_2^i),\ldots, T_{k_i}^i (w_{m_i}^i))$. Let $I=\{1,2,\ldots,l\}$ and $J$ be any subset of $I$ such that if $m_i\in J$ is odd then $m_i\ge n$. If $G$ has a (super) edge-magic labeling then $$(\sum_{j\in J} nG_i)+(\sum_{i\in I\setminus J}G_i^n)$$ has a (super) edge-magic labeling.
\end{tm}

\proof Let $j\in J$. If $m_i$ is even we define $r_i=m_i/2$. Otherwise, we define $r_i=(m_i+n)/2$. Thus, the
subgroup of $\mathbb{Z}_n$ generated by $m_i-2r_i$ is the trivial subgroup. Hence, inheriting the notation introduced in Theorem \ref{producte_uni}, und$(D_i\otimes_{h_i}\{C_n^+,C_n^-\})\cong nG_i$.
Now, let $i\in I\setminus J$. If $m_i$ is odd we define $r_i=(m_i-1)/2$. Otherwise, we define $r_i=m_i/2-1$. In both cases, the subgroup of $\mathbb{Z}_n$ generated by $m_i-2r_i$ is $\mathbb{Z}_n$. Thus, inheriting the notation introduced in Theorem \ref{producte_uni}, und$(D_i\otimes_{h_i}\{C_n^+,C_n^-\})\cong G_i^n$. Hence, by considering the function $h:E(D)\rightarrow \{C_n^+,C_n^-\}$ defined by $h(e)=h_i(e)$, for each $e\in D_i$ and $i=1,2,\ldots, l$, by Lemma \ref{components} we obtain that und$(D\otimes_h\{C_n^+,C_n^-\})\cong(\sum_{j\in J} nG_i)+(\sum_{i\in I\setminus J}G_i^n)$. Therefore, since $\{C_n^+,C_n^-\}\subset \mathcal{S}_n^{(5n+3)/2}$ the  result follows by Theorem \ref{super}.
\qed

Let $f:V(G)\cup E(G)\rightarrow \{i\}_{i=1}^{p+q}$ be a super edge-magic labeling of a $(p,q)$-graph $G$, with $p=q$. The {\it odd labeling} and {\it even labeling} obtained from $f$, denoted respectively by $o(f)$ and $e(f)$, are the labelings $o(f), e(f):V(G)\cup E(G)\rightarrow \{i\}_{i=1}^{p+q}$ defined as follows:
(i) on the vertices:
 $o(f)(x)=2f(x)-1$ and $e(f)(x)=2f(x)$, for all $x\in V(G)$,
(ii) on the edges:
 $o(f)(xy)=2\hbox{val}(f)-2p-2-o(f)(x)-o(f)(y)$ and
      $ e(f)(xy)=2\hbox{val}(f)-2p-1-e(f)(x)-e(f)(y)$, for all $xy\in E(G).$

\begin{lm}[\cite{LopMunRiu6}]\label{odd_even_labelings}
Let $G$ be a $(p,q)$-graph with $p=q$ and let $f:V(G)\cup E(G)\rightarrow \{i\}_{i=1}^{p+q}$ be a super edge-magic labeling of $G$. Then, the odd labeling $o(f)$ and the even labeling $e(f)$ obtained from $f$ are edge-magic labelings of $G$ with magic sums $\hbox{val}(o(f))=2\hbox{val}(f)-2p-2$ and $\hbox{val}(e(f))=2\hbox{val}(f)-2p-1$.
\end{lm}

Following an idea that appeared in \cite{LopMunRiu7}, we obtain the next result.

\begin{tm}\label{versio_complerta_i_magic_sums}
Let $D$ be any super edge-magic labeled digraph of order and size equal to $p$ and assume that the vertices take the name of their labels. Consider a set of functions $h_i$, where $h_i: E(D_i)\rightarrow \mathcal{S}_n^k$ and $D_i$ is defined recursively as follows: $D_1=D$ and $D_{i+1}=D_i\otimes_{h_i}\mathcal{S}_n^k$, for each $i\ge 2$. Then, $D_i$ admits at least $i+1$ edge-magic labelings with $i+1$ distinct edge-magic sums.
\end{tm}

\proof
We will prove the result by induction on $i$. First of all, notice that, by Theorem \ref{super} each digraph $D_i$ is super edge-magic, for $i\ge 2$. Also, notice that by Lemma \ref{odd_even_labelings} any super edge-magic labeled digraph of order and size equal to $p$ admits two edge-magic labelings with consecutive magic sums.
Thus, $D_1=D$ has at least $2$ edge-magic labelings with $2$ distinct edge-magic sums. Suppose now that $D_i$ admits at least $i+1$ edge-magic labelings with $i$ distinct edge-magic sums. By Corollary \ref{valences_differ}, the induced edge-magic labelings of $D_{i+1}$ (see the paragraph just before Proposition \ref{Pro_induced_labeling}),  differ at least by three units. Note that, by
definition of the $\otimes_h$-product, we have that $|V(D_{i+1})|=n|V(D_i)|$ and $|E(D_{i+1})|=n|E(D_i)|$, that is, equality $|V(D_i)|=|E(D_i)|$ holds for every $i\ge 1$. Hence, Lemma \ref{odd_even_labelings} implies that $D_{i+1}$ admits two edge-magic labelings with consecutive magic sums. Therefore, at least one of them is different from the ones induced by the $h_i$-product of $D_i$ with $\mathcal{S}_n^k$. That is, $D_{i+1}$ admits at least $i+1$ edge-magic labelings with $i+1$ distinct edge-magic sums.\qed

\begin{co}\label{SEM_magic_sums_suma_unicyclics}
Let $l,m,n$ and $s$ be nonnegative integers, such that, $m\ge 1$, $n\ge 3$ is odd and if $m$ is odd then $m\ge n$. Let $G=(T_1(w_1),T_2(w_2),\ldots, T_m(w_m))$ be a unicyclic graph that is super edge-magic. Let  $(a_0^l,a_1^l,\ldots, a_l^l)$ be a sequence of positive integers such that, $a_0^0=n$ when $l=0$,
$$\sum_{i=0}^l\frac{a_i^l}{n^{l-i}}=n \quad \hbox{and}\quad \sum_{i=0}^k\frac{a_i^l}{n^{k-i}}\in \mathbb{Z}^+,\ \hbox{for each}\ k=0,l-1, \ \hbox{when} \ l\ge 1.$$
Then, the graph
$\sum_{i=0}^la_i^l G^{n^{s+i}}$ is super edge-magic. Moreover, there exist at least $l+s+1$ diferent edge-magic labelings with at least $l+s+1$ different magic sums.
\end{co}

\proof By Theorem \ref{th_charact_decom_suma_unicyclics}, there exists a sequence of functions $h_i: E(D_i)\rightarrow \{C_n^+,C_n^-\}$, $i=1,2,\ldots, l+s$, where $D_1$ is an oriented graph obtained from $G$ by considering a strong orientation of its cycle and the digraphs $D_i$ are defined recursively as follows: $D_{i+1}=D_i\otimes_{h_i}\{C_n^+,C_n^-\}$, for each $i\in \{2,\ldots,l+s\}$, such that
$\sum_{i=0}^la_i^l G^{n^{s+i}}\cong \hbox{und}(D_{l+s+1}).$ Thus, by Theorem \ref{versio_complerta_i_magic_sums}, since $\{C_n^+,C_n^-\}\subset \mathcal{S}_n^{(5n+3)/2}$ the result follows.\qed

\noindent {\bf Acknowledgements} The research conducted in this document by the first author has been supported by the Spanish Research Council under project
MTM2011-28800-C02-01 and  by the Catalan Research Council
under grant 2009SGR1387.

\end{document}